\documentclass[reqno,12pt]{amsart}

\usepackage{amsmath}
\usepackage{amsfonts}
\usepackage{amssymb}
\usepackage{eufrak}
\usepackage{amscd}
\usepackage{amsthm}
\usepackage{epsfig}
\usepackage{amstext}
\usepackage[all]{xy}

\usepackage{tikz-cd}

\usepackage{amsmath}
\usepackage{amsfonts}
\usepackage{amssymb}

\usepackage{amsmath}
\usepackage{amsfonts}
\usepackage{amssymb}
\usepackage{eufrak}
\usepackage{amscd}
\usepackage{amsthm}
\usepackage{epsfig}
\usepackage{amstext}
\usepackage[all]{xy}

\theoremstyle{plain}

 \newtheorem{theorem}{Theorem}

 \newtheorem{proposition}[theorem]{Proposition}
 \newtheorem{corollary}[theorem]{Corollary}
 \newtheorem{lemma} [theorem] {Lemma}
 \newtheorem{remark}[theorem]{Remark}

\begin{document}

\title{Continuity of spectral radius and type I $C^*$-algebras}
\author{Tatiana Shulman}
\address{Department of Mathematical Physics and Differential Geometry, Institute of Mathematics of Polish Academy of Sciences, Warsaw}
\thanks{The research of the
 author was supported by the Polish National Science Centre grant
under the contract number DEC- 2012/06/A/ST1/00256, by the grant H2020-MSCA-RISE-2015-691246-QUANTUM DYNAMICS and Polish Government grant 3542/H2020/2016/2,  and from the Eric Nordgren
Research Fellowship Fund at the University of New Hampshire.}

\subjclass[2010]{46L05}



\keywords{Spectral radius, nilpotent elements, type I $C^*$-algebras, projectivity}

\begin{abstract}

It is shown that the  spectral radius is continuous on a $C^*$-algebra if and only if the $C^*$-algebra is type I.
This answers a question of V. Shulman and Yu.~Turovskii \cite{PapaYura}. It is shown  also
that the closure of nilpotents in a $C^*$-algebra contains an element with non-zero spectrum if and only if  the $C^*$-algebra is not
 type I.

 \end{abstract}
\maketitle

\section{Introduction}
Let $H$ be a Hilbert space and $B(H)$ the $C^*$-algebra of all bounded operators on $H$.
 As is well-known, the
  spectral radius is not a continuous function on  $B(H)$, it is only upper
  semi-continuous (see \cite{Halmos}, Solutions 103 and
  104).
However it is not hard to prove that on the $C^*$-algebra $K(H)$ of all compact operators it is continuous. Thus there arises a question on which $C^*$-algebras the spectral radius function is continuous. In \cite{PapaYura} V. Shulman and Yu. Turovskii proved that it is continuous on any type I $C^*$-algebra and asked if the converse is true (\cite{PapaYura}, Problem 9.20).
In this short paper we answer this question in the affirmative.

\medskip

{\bf Theorem A.} {\it Let $\mathcal A$ be a $C^*$-algebra. The spectral radius is a continuous function on $\mathcal A$ if and only if $\mathcal A$ is type I.}

\medskip

Theorem A adds one more (analytical) characterization to numerous characterizations of type I $C^*$-algebras.


One of  famous questions of P. Halmos asks: "What is the closure of the nilpotent
operators on a complex, separable, infinite-dimensional Hilbert space?"
It was completely
 answered  in \cite{AFV} (see also \cite{Herrero}). In particular  the answer shows that in the closure of nilpotent  operators
 there  are many
   operators with non-zero spectrum.  In the recent work  \cite{Skoufranis} P. Skoufranis  considered similar question for $C^*$-algebras with focus on normal limits of nilpotents.  His results show in particular that in many $C^*$-algebras (for example in all UHF-algebras and all purely infinite $C^*$-algebras) the closure of nilpotents contains  elements with non-zero spectrum. Here we give a characterization of such $C^*$-algebras.

\medskip

{\bf
Theorem  B.}  {\it
The closure of nilpotents in a $C^*$-algebra $\mathcal A$ contains an element with non-zero spectrum if and only if $\mathcal A$ is not type I.}

\section{Preliminaries}

\subsection{AF-telescopes and projectivity}

Let $A=\overline{\bigcup{A_n}}$ be an inductive limit of a sequence of $C^\ast$-algebras $$A_1 \subset A_2 \subset \ldots \subset A$$
with injective connecting maps. In \cite{LorPed}  the {\it mapping telescope} of $(A_n)$ was defined as the $C^\ast$-algebra
\begin{align*}
    T(A)=\{f\in C_0\left((0,\infty],A\right)|\; t\le n \Rightarrow f(t)\in A_n\}.
\end{align*}
Clearly the mapping telescope depends on the sequence $(A_n)$, but nevertheless we will use the notation $T(A)$. When each $A_n$ is finite-dimensional, the $C^\ast$-algebra $A$ is AF, and  $T(A)$ is called an {\it AF-telescope}.
 In particular, we denote by $T({M}_{2^\infty})$ the mapping telescope corresponding to the inductive sequence
\begin{align*}
    {M}_2\subset {M}_4\subset...\subset {M}_{2^n}\subset...\subset {M}_{2^\infty}
\end{align*}
where ${M}_{2^n}$ is identified with a subalgebra of ${M}_{2^{n+1}}$ by the map $a\mapsto a\oplus a$. Recall that ${M}_{2^\infty}$ is referred to as the CAR algebra (\cite{Davidson}).

\medskip

Recall that a $C^*$-algebra $A$ is {\it projective} (\cite{B1}, \cite{L}) if for any $C^\ast$-algebras $B$ and $C$ with surjective $\ast$-homomorphism $q:B\to C$, any $\ast$-homomorphism $\phi:A\to C$ lifts to a $\ast$-homomorphism $\psi:A\to B$ such that $q\circ\psi=\phi$. In other words, we have the following commutative diagram.

$$\xymatrix { & B \ar@{->>}[d]^-{q} \\ A  \ar[r]_-{\phi}  \ar@{-->}[ur]^{\psi}  & C}$$

In \cite{LorPed} Loring and Pedersen proved that AF-telescopes are projective. This fact will be crucial for the proof of Theorem B.  Some other applications of projectivity of AF-telescopes can be found in the recent work of Kristin Courtney and the author \cite{KristinTanya}.

\subsection{ Type I $C^*$-algebras}
Let $H$ be a Hilbert space.  A $C^*$-algebra $\mathcal A$ is {\it type I} (or, equivalently,  {\it GCR}) if $$K(H)\subseteq \pi(\mathcal A)$$ for any irreducible representation $(\pi, H)$ of $\mathcal A$.

A $C^*$-algebra $\mathcal D$ is a {\it subquotient} of a $C^*$-algebra $\mathcal A$ if there is a $C^*$-subalgebra $\mathcal B$ of $\mathcal A$ and a surjective $\ast$-homomorphism $q: \mathcal B \to \mathcal D$.

 We will need the following property (in fact a characterization) of non-type I $C^*$-algebras, which is due to Glimm for separable case \cite{Glimm} and Sakai for the general case \cite{Sakai}.

\begin{theorem}\label{GlimmTheorem} (Glimm, Sakai) If a $C^*$-algebra is non-type I, then it has a subquotient isomorphic to the CAR algebra.
\end{theorem}

\section{Proof of Theorems A and B}

We start by constructing a sequence of nilpotent elements in the CAR-algebra $\mathbb{M}_{2^\infty}$ that  converges to  an element with positive spectral radius.

Let $\epsilon_n = 1/{2^n}, $ $n\ge 0$.
Define matrices $A_n\in M_{2^n}$ as follows: $$A_1 = \left(\begin{array}{cc} 0 & 0\\ \epsilon_0 & 0\end{array}\right)$$

 $$A_2 = \left(\begin{array}{cccc} 0 & &  &\\ \epsilon_0 & 0 &&\\ &\epsilon_1&0&\\&&\epsilon_0 &0\end{array}\right)$$

 $$A_3 = \left(\begin{array}{cccccccc} 0 & &  & &&&&\\ \epsilon_0 & 0 &&&&&&\\ &\epsilon_1&0&&&&&\\&&\epsilon_0 &0&&&&\\
 &&&\epsilon_2 &0 &&&\\ &&&&\epsilon_0&0&&\\ &&&&&\epsilon_1&0&\\ &&&&&&\epsilon_0&0\end{array}\right)$$ and so on.
In other words, to obtain $A_{k+1}$ we take  $A_k\oplus A_k$ and put $\epsilon_k$ in the middle of the first diagonal under the main diagonal.

Then in the first diagonal under the main one $A_n$
 has the first $2^n-1$ elements of the so-called "the abacaba order" sequence
 \begin{equation}\label{zvezdochka}  \epsilon_0, \epsilon_1, \epsilon_0, \epsilon_2, \epsilon_0, \epsilon_1, \epsilon_0, \epsilon_3, \ldots
\end{equation} (here each second term equals $\epsilon_0$, each second of the remaining terms  equals  $\epsilon_1$, and so on)   and all the other entries are zeros.

 Below we will identify matrices from $M_{2^n}$ with their images in $M_{2^{n+1}}$ and in $M_{2^{\infty}}$.

\medskip

 The construction of the sequence of elements of $M(2^{\infty})$ above was inspired by Kakutani's construction of a  sequence of nilpotent operators in his remarkable proof of discontinuity
 of spectral radius on $B(H)$ (see \cite{Halmos}, Solution 104).

 \begin{lemma}\label{1} The sequence $A_n\in M_{2^{\infty}}$, $n\in \mathbb N$, converges to some $A\in M_{2^{\infty}}$ such that $\rho(A) >0$.
 \end{lemma}
 \begin{proof} We have $$\|A_1 - A_2\| = \| \left(\begin{array}{cccc} 0 & &  &\\ \epsilon_0 & 0 &&\\ &&0&\\&&\epsilon_0 &0\end{array}\right) -  \left(\begin{array}{cccc} 0 & &  &\\ \epsilon_0 & 0 &&\\ &\epsilon_1&0&\\&&\epsilon_0 &0\end{array}\right)\| = \| \left(\begin{array}{cccc} 0 & &  &\\ 0 & 0 &&\\ &\epsilon_1&0&\\&&0 &0\end{array}\right)\| = \epsilon_1,$$
 $$\|A_2 - A_3\| = \|\left(\begin{array}{cccccccc} 0 & &  & &&&&\\ \epsilon_0 & 0 &&&&&&\\ &\epsilon_1&0&&&&&\\&&\epsilon_0 &0&&&&\\
 &&& &0 &&&\\ &&&&\epsilon_0&0&&\\ &&&&&\epsilon_1&0&\\ &&&&&&\epsilon_0&0\end{array}\right) - \left(\begin{array}{cccccccc} 0 & &  & &&&&\\ \epsilon_0 & 0 &&&&&&\\ &\epsilon_1&0&&&&&\\&&\epsilon_0 &0&&&&\\
 &&&\epsilon_2 &0 &&&\\ &&&&\epsilon_0&0&&\\ &&&&&\epsilon_1&0&\\ &&&&&&\epsilon_0&0\end{array}\right)\| = \epsilon_2,$$ and, similarly, $$\|A_n - A_{n+1}\| = \epsilon_n.$$  Hence there is $A\in M(2^{\infty})$ such that \begin{equation}\label{eq0}A_n\to A.\end{equation}

 Let us denote by $\alpha_i$ the i-th element of the sequence (\ref{zvezdochka}). Then for each $n\in \mathbb N$ $$A_n = \left(\begin{array}{ccccc} 0&&&&\\\alpha_0&0&&&\\&\alpha_1& 0&&\\&&\ddots&&\\&&&\alpha_{2^n-2}&0\end{array}\right).$$ The corresponding basis vectors  will be denoted by $e_i$'s.

 Fix k. Then for any $n$ such that $2^n >k$ we have $$A_n^k e_i = \alpha_{i+k-2}\alpha_{i+k-3} \ldots \alpha_{i-1} e_{i+k},$$ when $i\le 2^n-k$ and $A_n^k e_i = 0$, when $i> 2^n - k.$ Hence $$\|A_n^k   \|= \max_{1\le i\le 2^n -k} |\alpha_{i+k-2}\alpha_{i+k-3} \ldots \alpha_{i-1} |.$$ Then \begin{equation}\label{eq1}\|A_n^k\|^{1/k}  = \max_{1\le i\le 2^n -k} |\alpha_{i+k-2}\alpha_{i+k-3} \ldots \alpha_{i-1} |^{1/k} \ge |\alpha_{k-1}\alpha_{k-2} \ldots \alpha_0|^{1/k}\end{equation} for any $n\in \mathbb N$. It was proved by Kakutani  [see \cite{Halmos}, Solution 104]  that
 \begin{equation}\label{eq2} \liminf_k |\alpha_{k-1}\alpha_{k-2} \ldots \alpha_0|^{1/k} > 0. \end{equation}
 By (\ref{eq0}), (\ref{eq1}) and (\ref{eq2}) $$\rho(A) = \lim_k\|A^k\|^{1/k} = \lim_k \lim_n \|A_n^k\|^{1/k} \ge \liminf_k |\alpha_{k-1}\alpha_{k-2} \ldots \alpha_0|^{1/k} > 0.$$

 \end{proof}

\begin{remark} There exist other constructions than the one above, which demonstrate discontinuity of spectral radius in $M(2^{\infty})$ (see \cite{Skoufranis}). However the one above is useful for the proof of Theorem B.
\end{remark}

 \bigskip

{\it Proof of Theorem B.}
 The "only if" part follows from continuity of spectral radius on type I $C^*$-algebras [\cite{PapaYura}, Cor. 9.18]. To prove the "if" part, assume $\mathcal A$ is not type I. It will be sufficient to construct a sequence of nilpotent elements in $\mathcal A$ converging to an element with a positive spectral radius.  By Theorem \ref{GlimmTheorem} there is a $C^*$-subalgebra $\mathcal B\subseteq \mathcal A$ and a surjective $\ast$-homomorphism $q: \mathcal B \to M(2^{\infty})$. Let $ev_{\infty}: T(M(2^{\infty})) \to M(2^{\infty})$ be the evaluation at infinity map.

$$\xymatrix { A  & B  \ar@_{{(}->}[l] \ar@{->>}[d]_{q}& \\  & M(2^{\infty}) & T(M(2^{\infty})) \ar[l]_{ev_{\infty}} \ar@{.>}[ul]_{\phi}&
}$$
Since AF-telescopes are projective (\cite{LorPed}), there is $\phi: T(M(2^{\infty})) \to B\subseteq A$ such that \begin{equation}\label{eq5}q\circ \phi = ev_{\infty}.\end{equation}

 Let $A_n$'s and $A$ be as in Lemma \ref{1}. For each $t\in [k, k+1]$, $k\in \mathbb N$, let $A_t$ be the linear path connecting $A_k$ and $A_{k+1}$, and for each $t\in (0, 1]$ let $A_t$ be the linear path connecting $0$ and $A_{1}$. Define $f_n$, $n\in \mathbb N$, and $f$ by
$$f_n(t) = \begin{cases} A_t, & t\in [k, k+1], k\le n \\ A_n, & t\ge n.
\end{cases}$$ and $$f(t) = \begin{cases} A_t, & t\in [k, k+1] \\ A, & t = \infty.
\end{cases}$$ The functions $f_n$, $n\in \mathbb N$, are obviously continuous, and since $A_n \to A$, $f$ is also continuous. Thus $f_n$ and $f$ belong to $T(M((2^{\infty}))$. Since all $A_n$'s are lower-triangular, so are $A_t$, for each $t\in (0, \infty).$ Hence for each $n$ and for each $t$, $f_n(t)$ is nilpotent of order $2^n$. Hence for any $n\in \mathbb N$, $f_n$ is nilpotent. Clearly   $f_n \to f$ and by Lemma \ref{1}
\begin{equation}\label{new}\rho(f) \ge \rho(f(\infty)) = \rho(A) >0.\end{equation}
Let $$a_n = \phi(f_n), \; a = \phi(f).$$ Then $a_n\to a$ and each $a_n$ is nilpotent. By (\ref{eq5}), $q(a) = f(\infty)$. From this and (\ref{new}) we obtain  $$\rho(a) \ge \rho(q(a)) = \rho(f(\infty)) >0.$$

\qed

\bigskip

{\it Proof of Theorem A.} The "if" part was proved in [\cite{PapaYura}, Cor.9.18].  The "only if" part follows from Theorem B. \qed

\medskip

\begin{corollary} Let $\mathcal A$ be a $C^*$-algebra. The following are equivalent:

1) Each precompact set of $\mathcal A$ is a point of continuity of the joint spectral radius,

2) $\mathcal A$ is type I.
\end{corollary}
\begin{proof} 2) $\Rightarrow$ 1) is proved in [\cite{PapaYura}, Cor. 10.31].

1) $\Rightarrow$ 2) follows from Theorem A.
\end{proof}

\bigskip

It would be interesting to characterize the class of $C^*$-algebras for which the closure of nilpotents contains a normal element.
The following is an easy observation.

\begin{proposition} If the closure of nilpotents in a $C^*$-algebra $\mathcal A$
contains a normal element, then $\mathcal A$ is not residually type I.
\end{proposition}
(By residually type I $C^*$-algebra we mean a $C^*$-algebra that has a separating family of $\ast$-homomorphisms into type I $C^*$-algebras).
\begin{proof} Suppose $N\in \mathcal A$ is normal, $a_n\in \mathcal A$ are nilpotents and $a_n \to N$. Then, by Theorem A (in fact by its "if" part which is proved in \cite{PapaYura}),   for any $\ast$-homomorphism $\rho$ to a type I $C^*$-algebra,
$\rho(N)$ is  quasinilpotent. Since it is also normal, we conclude that $\rho(N)=0$.
\end{proof}

\end{document}